\documentclass[preprint,12pt]{amsart}
\usepackage{tikz}
\usetikzlibrary{matrix}

\usepackage{amssymb}
\usepackage{verbatim}
\usepackage{array}
\usepackage{latexsym}
\usepackage{enumerate}
\usepackage{amsmath}
\usepackage{amsfonts}
\usepackage{amsthm}
\usepackage{color}
\usepackage[english]{babel}

\newtheorem{theorem}{Theorem}[section]
\newtheorem{proposition}[theorem]{Proposition}
\newtheorem{lemma}[theorem]{Lemma}
\newtheorem{corollary}[theorem]{Corollary}

\newtheorem{condition}[theorem]{Condition}
\newtheorem{question}[theorem]{Question}
{\theoremstyle{definition}
\newtheorem{definition}{Definition}}
\newtheorem{rem}[theorem]{Remark}

\def\S{\mathbf S}

\def\cC{\mathcal C}

\def\cX{\mathcal X}
\def\cY{\mathcal Y}
\def\cZ{\mathcal Z}

\def\Aut{\mbox{\rm Aut}}
\def\K{\mathbb{K}}
\def\PG{{\rm{PG}}}

\def\deg{\mbox{\rm deg}}

\def\div{\mbox{\rm div}}

\def\Aut{\mbox{\rm Aut}}

\def\Div{\mbox{\rm div}}

\def\dim{\mbox{\rm dim}}
\def\supp{\mbox{\rm Supp}}

\def\gg{\mathfrak{g}}

\newcommand{\PSL}{\mbox{\rm PSL}}
\newcommand{\PGL}{\mbox{\rm PGL}}

\newcommand{\PSU}{\mbox{\rm PSU}}
\newcommand{\PGU}{\mbox{\rm PGU}}

\newcommand{\aut}{\mbox{\rm Aut}}

\newcommand{\g}{\gamma}

\def\supp{{\rm Supp}}

\newcommand{\ha}{{\textstyle\frac{1}{2}}}

\title{Large automorphism groups of ordinary curves of even genus in odd characteristic}
\date{}

\author{Maria Montanucci}
\address{Department of Applied Mathematics and Computer Science\\Technical University of Denmark\\ Asmussens Allé \\ 2800 Kongens Lyngby (Denmark)}
\email{marimo@dtu.dk}

\author{Pietro Speziali}
\address{Instituto de Ci\^encias Matem\'aticas e de Computa\c{c}\~ao,\\ Universidade de S\~ao Paulo,\\S\~ao Carlos, SP 13560-970, (Brazil)}
\email{pietro.speziali@icmc.usp.br}

\begin{document}
\maketitle
\begin{abstract}
Let $\cX$ be a (projective, non-singular, geometrically irreducible) curve of even genus $g(\cX) \geq 2$ defined over an algebraically closed field $K$ of odd characteristic $p$. If the $p$-rank $\gamma(\cX)$ equals $g(\cX)$, then $\cX$ is \emph{ordinary}. In this paper, we deal with \emph{large} automorphism groups $G$ of ordinary curves of even genus.  We prove that $|G| < 821.37g(\cX)^{7/4}$.  
The proof of our result is based on the classification of automorphism groups of curves of even genus in positive characteristic, see \cite{giulietti-korchmaros-2017}. According to this classification,
for the exceptional cases $\aut(\cX) \cong {\rm Alt}_7$ and $\aut(\cX) \cong \rm{M}_{11}$ we show that the classical Hurwitz bound $|\aut(\cX)| < 84(g(\cX)-1)$ holds, unless $p=3$, $g(\cX)=26$ and $\aut(\cX) \cong \rm{M}_{11}$; an example for the latter case being given by the modular curve $X(11)$ in characteristic $3$.
\end{abstract}



\section{Introduction}\label{intr}
Let $p >0$ be a prime. By a \emph{curve} $\cX$, we mean a  projective, non-singular, geometrically irreducible curve defined over an algebraically closed field $K$ of characteristic $p$. Usually, the study of the geometry of $\cX$ is carried out through the study of its birational invariants, such as its \emph{genus} $g(\cX)$, its $p$-\emph{rank} (or \emph{Hasse-Witt invariant}) $\gamma(\cX)$, and its \emph{automorphism group} $\aut(\cX)$. The genus can be thought of as the dimension of the $K$-vector space of holomorphic differentials on $\cX$, while the $p$-rank is the dimension of the $K$-space  of holomorphic logarithmic differentials. A curve $\cX$ is \emph{ordinary} if $g(\cX) = \gamma(\cX)$. An automorphism  of $\cX$ arises by defining the concept of a field automorphism $\sigma$ of the function field $K(\cX)$ fixing the ground field $K$ elementwise. 
 From a purely geometric point of view, $\cX$ admits a non-singular model in some projective space $\PG(r,K)$ for some $r\leq g(\cX)$, and every automorphism of $\cX$ can be represented by a linear collineation in $\PGL(r+1,K)$ leaving $\cX$ invariant. 
 
 Henceforth, we shall denote the automorphism group of $\cX$ by $\aut(\cX)$. 
 By a classical result, $\aut(\cX)$ is finite whenever $g(\cX) \geq 2$.  Also, if the characteristic $p$ of $K$ divides $|\aut(\cX)|$, then the classical Hurwitz bound 
 $$
 |\aut(\cX)| \leq 84(g(\cX)-1)
 $$
 no longer holds in general. 

For ordinary curves, Nakajima in his seminal work \cite{nakajima 1987} proved that 
$$
 |\aut(\cX)| \leq 84g(\cX)(g(\cX)-1). 
$$

However, Nakajima's bound seems to be far from optimal, and a bound of size at most $g(\cX)^{8/5}$ is expected; see \cite{KR}. If $\aut(\cX)$ is solvable and $p >2$, an even tighter bound 

\begin{equation}\label{kmbound}
|\aut(\cX)| \leq 34(g(\cX)+1)^{3/2}
\end{equation}
holds; see \cite {montanucci- korchmaros}. The latter bound is sharp up to the constant term; see \cite{montanucci-korchmaros-speziali}. More recently, the bound \ref{kmbound} has been proven to hold also for $p = 2$ and $g(\cX)$ even; see \cite{montanucci-speziali-JA}. The aforementioned results depend of the fact that \emph{large} solvable subgroups of $\aut(\cX)$ have a prescribed structure; see \cite[Lemma 2.1]{montanucci- korchmaros} and \cite[Lemma 4.1]{giulietti-korchmaros-2017}.

In this paper, we deal with ordinary curves of even genus. One may ask why, as conditions on the parity of the genus may look somewhat exotic. Loosely speaking, the hypothesis that $2 | g(\cX)$ gives strong restrictions on the structure of a Sylow $2$-subgroup of $\aut(\cX)$, whence the powerful tools from finite group theory can be exploited.
 

 In Section \ref{sec:ordevengenus}, we prove that if $\cX$ is an ordinary curve of even genus and $p$ is odd, then the following bound
 $$
 |\aut(\cX)| < 821.37g(\cX)^{7/4}
 $$
 holds, see Theorem \ref{thm:sec3main}. We point out that this bound can be refined according to the structure of $\aut(\cX)$. The possibilities for $\aut(\cX)$ are listed in Remark \ref{possibilita}. According to Lemma \ref{lemA24agos2015}, two \emph{exceptional} cases then arise: $\aut(\cX) \cong \rm{Alt_7}, M_{11}$, where $\rm{M}_{11}$ is the Mathieu group of degree $11$ while $\rm{Alt_7}$ denotes the alternating group of degree $7$. In these cases, we prove that the classical Hurwitz bound for $|\Aut(\cX)|$ holds, unless $p = 3$,  $g(\cX) = 26$ and $\aut(\cX) = \rm{M}_{11}$, see Propositions \ref{alt7} and \ref{m11}.  An example of an ordinary curve of genus $26$ admitting $\rm{M}_{11}$ as an automorphism group in characteristic $3$ is given by the modular curve $X(11)$, see Section \ref{sec:modular}. 
\section{Background and Preliminary Results}\label{sec2}
Our notation and terminology are standard. Well-known references for the theory of curves and algebraic function fields are \cite{hirschfeld-korchmaros-torres2008} and \cite{stbook}. Let $\cX$ be a curve defined over an algebraically closed field $K$ of positive characteristic $p$ for some prime $p$. We denote by $K(\cX)$ the function field of $\cX$. By a point $P \in \cX$ we mean a point in a nonsingular model of $\cX$; in this way, we have a one-to-one correspondence between points of $\cX$ and places of $K(\cX)$. 

Let $\aut(\cX)$ denote the full automorphism group of $\cX$. For a subgroup $S$ of $\aut(\cX)$, we denote by $K(\cX)^S$ the fixed field of $S$. A nonsingular model $\bar{\cX}$ of  $K(\cX)^S$ is referred as the quotient curve of $\cX$ by $S$ and denoted by $\cX/S$. 
The field extension $K(\cX):K(\cX)^S$ is Galois with Galois group $S$. For a point $P \in \cX$, $S(P)$ is the orbit of $P$ under the action of $S$ on $\cX$ seen as a point-set. The orbit $S(P)$ is said to be long if $|S(P)| = |S|$, short otherwise. There is a one-to-one correspondence between short orbits and ramified points in the extension $K(\cX):K(\cX)^S$. It might happen that $S$ has no short orbits; if this is the case, the cover $\cX \rightarrow \cX/S$ (or equivalently, the extension $K(\cX):K(\cX)^S$) is unramified. 

For $P \in \cX$, the subgroup $S_P$ of $S$ consisting of all elements of $S$ fixing $P$ is called the stabilizer of $P$ in $S$.  We will often refer to $S_P$ as to the \emph{1-point stabilizer} of $S$. For a non-negative integer $i$, the $i$-th ramification group of $\cX$ at $P$ is denoted by $S_P^{(i)}$, and defined by
$$
S_P^{(i)}=\{\sigma \ | \ v_P(\sigma(t)-t)\geq i+1, \sigma \in S_P\}, 
$$
 where $t$ is a local parameter at $P$ and $v_P$ is the respective discrete valuation. Here $S_P=S_P^{(0)}$. Furthermore, $S_P^{(1)}$ is the unique Sylow $p$-subgroup of $S_P^{(0)}$, and the factor group $S_P^{(0)}/S_P^{(1)}$ is cyclic of order prime to $p$; see \cite[Theorem 11.74]{hirschfeld-korchmaros-torres2008}. In particular, if $S_P$ is a $p$-group, then $S_P=S_P^{(0)}=S_P^{(1)}$. 
 
 The following non-standard notation introduced by Nakajima in \cite{nakajima 1987} can be useful.  Let us denote the covering $\cX \rightarrow \cX/S $ by $\pi$. For a point $Q \in  \cX/S$, we define  the ramification index $e_Q$  and the different exponent $d_Q$ as follows: take $P \in \cX$ which satisfies $\pi(P) = Q$. Then, $e_Q = |S^{(0)}_P|$ and $d_Q = \sum_{i = 0}^{\infty}(|S_P^{(i)}|- 1)$. Note that, since $\pi$ is a Galois covering, $e_Q$ and $d_Q$ do not depend on the choice of $P$.

Let $g$ and $\bar{g}$ be the genus of $\cX$ and $\bar{\cX}=\cX/S$, respectively. The Hurwitz genus formula is 
\begin{equation}\label{rhg}
2g-2=|S|(2\bar{g}-2)+\sum_{P \in \cX}\sum_{i \geq 0}\big(|S_P^{(i)}|-1\big);
\end{equation}
see \cite[Theorem 11.72]{hirschfeld-korchmaros-torres2008}.
 If $\ell_1,\ldots,\ell_k$  are the sizes of the short orbits of $S$, then (\ref{rhg}) yields
\begin{equation}\label{rhso}
2g-2 \geq |S|(2\bar{g}-2)+\sum_{\nu=1}^{k} \big(|S|-\ell_\nu\big),
\end{equation}
and equality holds if $\gcd(|S_P|,p)=1$ for all $P \in \cX$; see \cite[Theorem 11.57 and Remark 11.61]{hirschfeld-korchmaros-torres2008}.

Let $\gamma = \gamma(\cX)$ denote the $p$-rank (equivalently, the Hasse-Witt invariant) of $\cX$. If $S$ is a $p$-subgroup of $\aut(\cX)$ then
the Deuring-Shafarevich formula, see \cite[Theorem 11.62]{hirschfeld-korchmaros-torres2008}, states that
\begin{equation}
    \label{eq2deuring}
\gamma-1={|S|}(\bar{\gamma}-1)+\sum_{i=1}^k (|S|-\ell_i),
    \end{equation}
where $\bar{\gamma} = \gamma(\cX/S)$ is the $p$-rank of $\cX/S$ and $\ell_1,\ldots,\ell_k$ denote the sizes of the short orbits of $S$. Both the Hurwitz and Deuring-Shafarevich formulas hold true for rational and elliptic curves provided that $G$ is a finite subgroup.

A subgroup of $\aut(\cX)$ is a \emph{$p'$-group} (or \emph{a prime to $p$} group) if its order is prime to $p$. 
A subgroup $G$ of $\aut(\cX)$ is \emph{tame} if the $1$-point stabilizer of any point in $G$ is $p'$-group. Otherwise, $G$ is \emph{non-tame} (or \emph{wild}). 
If $G$ is tame, then the classical Hurwitz bound $|G|\leq 84(g(\cX)-1)$ holds, but for non-tame groups this is far from being true. 

In this paper we deal with \emph{ordinary curves}, that is, curves for which equality $g(\cX) = \gamma(\cX)$ holds. We now collect some results regarding automorphism groups of ordinary curves. 
\begin{theorem}[\cite{nakajima 1987}, Theorem 3] \label{thm84g2}
Let $\cX$ be an ordinary curve with $g(\cX) \geq 2$. Then the following inequality 
\begin{equation}
|\aut(\cX)| \leq 84(g(\cX)-1)g(\cX)
\end{equation}
holds.
\end{theorem}

\begin{theorem}[\cite{nakajima 1987}, Theorem 2(i)]\label{2i}
Let $\cX$ be ordinary, and let $G$ be a finite subgroup of $\aut(\cX)$. Then for every point $P$ of $\cX$,  $G_P^{(2)}$ is trivial. 
\end{theorem}

We end this section recalling some definitions from Group Theory. 

\begin{definition}
A subgroup $H$ of a group $G$ is said to be a \emph{minimal normal subgroup} if $H$ is normal in $G$ and any normal subgroup of $G$ properly contained in $H$ is trivial.
If $G$ is solvable then a minimal normal subgroup $H$ of $G$ is elementary abelian, see \cite[Theorem 5.46]{machi}.
\end{definition}

\begin{definition}
For a group $G$, the \emph{odd core} $O(G)$ is its maximal normal subgroup of odd order. A group is \emph{odd core-free} if $O(G)$ is trivial. 
\end{definition}

\section{On ordinary curves of even genus in odd characteristic}\label{sec:ordevengenus}
In this section, $K$ is an algebraically closed field of odd characteristic $p$, $\cX$ is an ordinary curve 
 $g=g(\cX) \geq 2$, and $G$ is a subgroup of $\aut(\cX)$. 

Our aim here is to show that $|G| < 821.37g(\cX)^{7/4}$.


First, we give some background. We start with  bounds as well as conditions  on the structure of the Sylow $p$-subgroups of solvable automorphism groups of ordinary curves. 


\begin{theorem}\emph{\cite[Theorem 1.1]{montanucci- korchmaros}} \label{solvable} 
\label{the1zi} Let $\cX$ be an ordinary  curve of genus $g \geq 2$ defined over $K$. 
If $G$ is a solvable subgroup of $\aut(\cX)$  or $G$ contains an elementary abelian minimal normal subgroup, then $|G|\leq 34(g+1)^{3/2}<68 \sqrt{2}g^{3/2}$.
\end{theorem}

The following is a consequence of the proofs of  \cite[Proposition 3.1 and Proposition 3.2]{montanucci- korchmaros}. 

\begin{theorem} \label{sylow} 
Let $\cX$ be an ordinary curve of genus $g \geq 2$ defined over $K$. If $G$ is a solvable subgroup of $\aut(\cX)$  or $G$ contains an elementary abelian minimal normal subgroup, then a Sylow $p$-subgroup of $G$ is also elementary abelian.
\end{theorem}

%

\begin{lemma}\emph{ \cite[Lemma 4.1]{giulietti-korchmaros-2017}}
\label{22dic2015} Let $G$ be a solvable automorphism group of an algebraic curve $\cX$ of genus $g \ge 2$ containing a normal subgroup $Q$ of odd order $d^k$ with $k \geq 1$ such that $|Q|$ and $[G:Q]$ are coprime. Suppose that a complement $U$ of $Q$ in $G$ is abelian, and that $N_G(U)\cap Q=\{1\}$.  If
\begin{equation}
\label{eq22bisdic2015}
{\mbox{$|G|\geq 30(g-1)$}},
\end{equation}
then  $d=p$ and $U$ is cyclic. Moreover, the quotient curve $\bar{\cX}=\cX/Q$ is rational and either
\begin{itemize}
\item[\rm(i)]  $\cX$ has positive $p$-rank, $Q$ has exactly two (non-tame) short orbits, and they are also the only short orbits of $G$; or
\item[\rm(ii)] $\cX$ has zero $p$-rank and $G$ fixes a point.
\end{itemize}
\end{lemma}

Henceforth, by $\cX$ we shall denote an ordinary curve of even genus.

The latter condition implies that the structure of $\aut(\cX)$ can be described in terms of its minimal normal subgroups and their Sylow $2$-subgroups. The following lemma deals with the non-solvable case. Here, a semi-dihedral group $SD_{2^h}$ of order $2^h$ is the finite group given by the following presentation: 
$$
SD_{2^h} = \langle a, x | a^{2^{h-1}} = x^2 = id, \  xax = a^{2^{h-2}-1} \rangle. 
$$

\begin{lemma}\emph{\cite[Lemma 6.7 and Lemma 6.11]{giulietti-korchmaros-2017}}
\label{lemA24agos2015} Let $G$ be  an automorphism group of a curve of even genus in characteristic $p >2$. If $G$ is odd core-free and it has a non-abelian minimal normal subgroup $N$ then $N$ is a simple group and one the following cases occurs for some prime power $q=d^k$ with $k$ odd.
\begin{itemize}
\item[\rm(i)] $N\cong \PSL(2,q)$ with $q\geq 5$ odd, a Sylow $2$-subgroup of $N$ is dihedral, and \ \\ $G \leq \Aut(N) \cong {\rm{P}}\Gamma{\rm {L}}(2,q)$. Hence, $\PSL(2,q)\le G \le {\rm{P}}\Gamma{\rm {L}}(2,q)$ with $q\ge 5$ odd; 

\item[\rm(ii)] $N\cong \PSL(3,q)$ with $q\equiv 3 \pmod 4$, a Sylow $2$-subgroup of $N$ is semidihedral, and \ \\ $G \leq \Aut(N) \cong {\rm{P}}\Gamma{\rm {L}}(3,q)$. Hence, $\PSL(3,q)\le G \le {\rm{P}}\Gamma {\rm{L}}(3,q)$ with $q\equiv 3 \pmod 4$;

\item[\rm(iii)] $N\cong \PSU(3,q)$ with $q\equiv 1 \pmod 4$, and a Sylow $2$-subgroup of $N$ is semidihedral, and \ \\ 
$G \leq \Aut(N) \cong {\rm{P}}\Gamma{\rm {U}}(3,q)$. Hence $\PSU(3,q)\le G \le {\rm{P}}\Gamma {\rm{U}}(3,q)$ with $q\equiv 1 \pmod 4$;

\item[\rm(iv)] $N\cong {\rm{Alt}_7}$, a Sylow $2$-subgroup of $N$ is dihedral, and $G=N \cong \Aut(N) \cong {\rm{Alt}_7}$;
\item[\rm(v)]  $N\cong \rm{M}_{11}$, the Mathieu group on $11$ letters, a Sylow $2$-subgroup of $N$ is semidihedral, and $G=N  \cong \Aut(N) \cong \rm{M}_{11}$.
\end{itemize}
\end{lemma}

\begin{rem}\emph{\cite[Remark 6.12]{giulietti-korchmaros-2017}}. \label{possibilita} {\em{The subgroups of ${\rm{P}}\Gamma {\rm{L}}(2,q)$ containing $\PSL(2,q)$ whose Sylow $2$-subgroups have $2$-rank $2$ are $\PSL(2,q),\PGL(2,q)$ and, when $q=d^k$ with an odd prime $d$ and $k\ge 2$, the semidirect product of $\PSL(2,q)$ or $\PGL(2,q)$ with a cyclic group whose order is an odd divisor of $k$. Analogous results are valid for $\PSL(3,q)$ and $\PSU(3,q)$. These groups are all the candidates for $G$ in (i),(ii),(iii) of Lemma \ref{lemA24agos2015}. We point out that in general it is not true that $d$ must be equal to $p$. For instance, the Wyman's sextic of equation  $x^6+y^6+1+(x^2+y^2+1)(x^4+y^4+1)-12x^2y^2 = 0$ in characteristic $17$ is ordinary and admits the symmetric group $\rm{Sym}_5 \cong \PGL(2,5)$ as an automorphism group.  
}}
\end{rem}

\begin{rem} Lemma \ref{lemA24agos2015} and Remark \ref{possibilita} give a complete description of odd core-free automorphism groups $G$ of an algebraic curve of even genus $g \geq 2$ admitting no solvable minimal normal subgroup.  In particular, they give the complete list of possibilities when $G$ itself is not solvable. Note that also the case $G$ simple is analyzed. In fact, if $G$ is simple then $G=N$ in Lemma \ref{lemA24agos2015}.  \end{rem}

Because of Theorem \ref{solvable}, some conditions on $G$ can be assumed without loss of generality.

\begin{condition} \label{assunta}
The automorphism group $G$ is odd core-free and $G$ has no elementary abelian minimal normal subgroups.
\end{condition}

Condition \ref{assunta} can be justified as follows. Assume that $O(G)$ is non-trivial. Thus, since $|O(G)|$ is odd, by the Feit-Thompson Theorem \cite[page 775]{feitt}, $O(G)$ is solvable. Then, by definition of minimality for normal subgroups, there exists a minimal normal subgroup $N$ of $G$ which is contained in $O(G)$. Since $O(G)$ is solvable $N$ is elementary abelian. If $G$ has an elementary abelian minimal normal subgroup then $|G| \leq 34(g+1)^{3/2}<68 \sqrt{2} g^{3/2}$ by Theorem \ref{solvable}.

Lemma \ref{22dic2015} has been improved for the particular case of ordinary curves in \cite[Lemma 2.12]{montanucci- korchmaros}.

\begin{lemma} \label{terribile_ord}
 Let $\cX$ be an ordinary algebraic curve of genus $g \ge 2$ defined over an algebraically closed field of odd characteristic $p$. Let $H$ be a solvable automorphism group of $\aut(\cX)$ containing a normal $p$-subgroup $Q$ such that $|Q|$ and $[H:Q]$ are coprime. Suppose that a complement $U$ of $Q$ in $H$ is abelian and that
\begin{equation}
\label{eq22bisdic2015}
|H|>  \left\{
\begin{array}{lll}
{\mbox{$18(g-1)$ for $|U|=3$}}, \\
{\mbox{$12(g-1)$ otherwise}}.
\end{array}
\right.
\end{equation}
Then $U$ is cyclic, and the quotient curve $\bar{\cX}=\cX/Q$ is rational. Furthermore, $Q$ has exactly two (non-tame) short orbits, say $\Omega_1, \Omega_2$. They are also the only short orbits of $H$, and $\gg(\cX)=|Q|-(|\Omega_1|+|\Omega_2|)+1.$
\end{lemma}

At this point further information about the complement $U$ in Lemma \ref{terribile_ord} can be obtained.

\begin{proposition} \label{complemento}
 Let $\cX$ be an ordinary algebraic curve of genus $g \ge 2$ defined over an algebraically closed field $\K$ of odd characteristic $p$. Let $G$ be a solvable automorphism group of $\aut(\cX)$ containing a minimal normal $p$-subgroup $Q$ which is also the maximal normal $p$-subgroup of $G$. Suppose also that $|Q|$ and $[H:Q]$ are coprime and that a complement $U$ of $Q$ in $G$ is abelian. If
\begin{equation}
|G|>  \left\{
\begin{array}{lll}
{\mbox{$18(g-1)$ for $|U|=3$}}, \\
{\mbox{$12(g-1)$ otherwise}},
\end{array}
\right.
\end{equation}
then $G=Q \rtimes U$ , $U$ is cyclic,  $|U| < \sqrt{4 g + 4}$ and the quotient curve $\bar{\cX}=\cX/Q$ is rational. Furthermore, $Q$ has exactly two (non-tame) short orbits, say $\Omega_1, \Omega_2$. They are also the only short orbits of $G$, and $g=|Q|-(|\Omega_1|+|\Omega_2|)+1.$
\end{proposition} 

\begin{proof}
From Lemma \ref{terribile_ord}, we only need to prove that $|U| < \sqrt{4 g + 4}$.

We know that $\cX/Q$ is rational, $G=Q\rtimes U$ where $U$ is cyclic, $Q$ has exactly two (non-tame) orbits, say $\Omega_1$ and $\Omega_2$, and they are also the only short orbits of $G$. 
We may also observe that $G_P$ with $P\in \Omega_1$ contains a subgroup $V$ isomorphic to $U$. In fact, $|Q||U|=|G|=|G_P||\Omega_1|=|Q_P\rtimes V||\Omega_1|=|V||Q_P||\Omega_1|$ with a prime to $p$ subgroup $V$ fixing $P$, whence $|U|=|V|$. Since $V$ is cyclic the claim follows.

We go on with the case where both $\Omega_1$ and $\Omega_2$ are nontrivial, that is, their lengths are at least $2$.

We note that $Q$ is abelian. Indeed if this would not be the case then $Z(Q)$, which is not trivial, is a normal subgroup of $G$ properly contained in $G$, a contradiction. In particular, $|Q|\leq 4\gg+4$ by \cite[Theorem 11.79]{hirschfeld-korchmaros-torres2008}. Also, either $|Q_P|$, or $|Q_R|$ is at most $\sqrt{4\gg+4}$. From Theorem \ref{2i}, $G_P^{(2)}$ at $P\in \Omega_1$ is trivial. Furthermore, for $G_P=Q_P\rtimes V$, and from \cite[Proposition 1]{nakajima 1987} $|U|=|V|\leq |Q_P|-1$. Hence
$|U|<|Q_P|\leq \sqrt{|Q|}\leq  \sqrt{4\gg+4}$ which is our claim. 

Suppose next $\Omega_1=\{P\}$ and $|\Omega_2| \geq 2$. 

Then $G$ fixes $P$, and hence $G=Q \rtimes U$ with an elementary abelian $p$-group $Q$.
Furthermore, $G$ has a permutation representation on $\Omega_2$ with kernel $K$. As $\Omega_2$ is a short orbit of $Q$, the stabilizer $Q_R$ of $R\in \Omega_2$ in $Q$ is nontrivial. Since $Q$ is abelian, this yields that $K$ is nontrivial, and hence it is a nontrivial elementary abelian normal subgroup of $G$. In other words, $Q$ is an $r$-dimensional vector space $V(r,p)$  over a finite field $\mathbb{F}_p$ with $|Q|=p^r$, the action of each nontrivial element of $U$ by conjugacy is a nontrivial automorphism of $V(r,p)$, and $K$ is a $U$-invariant subspace. 

By \cite[Theorem 6.1]{machi}, $K$ has a complementary $U$-invariant subspace. Therefore, $Q$ has a a non-trivial subgroup $M$ such that $Q=K\times M$, and $M$ is a normal subgroup of $G$. From the minimality of $Q$, $M=Q$ which is not possible.

We are left with the case where both short orbits of $Q$ are trivial. Our goal is to prove a much stronger bound for this case, namely $|U|\leq 2$ whence
\begin{equation}
\label{eq10oct2016} |G|\leq 2(g+1).
\end{equation}
We also show that if equality holds then $\cX$ is a hyperelliptic curve with equation
\begin{equation}
\label{eq010oct2016} f(U)=aT+b+cT^{-1},\,\,a,b,c\in \mathbb{K}^*,
\end{equation}
where $f(U)\in \mathbb{K}[U]$ is an additive polynomial of degree $|Q|$.

Let $\Omega_1=\{P_1\}$ and $\Omega_2=\{P_2\}$. Then $Q$ has two fixed points $P_1$ and $P_2$ but no nontrivial element in $Q$ fixes a point of $\cX$ other than $P_1$ and $P_2$.
Since from Lemma \ref{terribile_ord}, $g=|Q|-(|\Omega_1|+|\Omega_2|)+1$, we have
\begin{equation} \label{eq4n}
\gg(\cX)+1=\gamma(\cX)+1=|Q|.
\end{equation}
Therefore, $|U|\leq g$.
Actually, for our purpose, we need a stronger estimate, namely $|U|\le 2$. To prove the latter bound, we use some ideas from Nakajima's paper \cite{nakajima 1987} regarding the Riemann-Roch spaces $\mathcal{L}(\mathbf{D})$ of certain divisors $\mathbf{D}$ of $\mathbb{K}(\cX)$. Our first step is to show
\begin{enumerate}
\item[(i)] $\dim_{\mathbb{K}} \mathcal{L}((|Q|-1)P_1)=1$,
\item[(ii)] $\dim_{\mathbb{K}} \mathcal{L}((|Q|-1)P_1+P_2)\geq 2$.
\end{enumerate}
Let $\ell \geq 1$ be the smallest integer such that $\dim_{\mathbb{K}} \mathcal{L}(\ell P_1)=2$, and take $x \in \mathcal{L}(\ell P_1)$ with $v_{P_1}(x)=-\ell$. As $Q=Q_{P_1}$, the Riemann-Roch space $\mathcal{L}(\ell P_1)$ contains all $c_\sigma=\sigma(x)-x$ with $\sigma\in Q$. This yields $c_\sigma\in \mathbb{K}$ by $v_{P_1}(c_\sigma) \geq -\ell+1$ and our choice of $\ell$ to be minimal. Also, $Q=Q_{P_2}$ together with $v_{P_2}(x) \geq 0$ show $v_{P_2}(c_\sigma) \geq 1$. Therefore $c_\sigma =0$ for all $\sigma \in Q$, that is, $x$ is fixed by $Q$. From
$\ell=[\mathbb{K}(\cX):\mathbb{K}(x)]=[\mathbb{K}:\mathbb{K}(\cX)^Q][\mathbb{K(\cX)}^Q:\mathbb{K}(x)]$ and $|Q|=[\mathbb{K}:\mathbb{K}(\cX)^Q]$, it turns out that $\ell$ is a multiple of $|Q|$. Thus $\ell>|Q|-1$ whence (i) follows. From the Riemann-Roch theorem, $\dim_{\mathbb{K}}\mathcal{L}((|Q|-1)P_1+P_2)\geq |Q|-g+1=2$ which proves (ii).

Let $d \geq 1$ be the smallest integer such that $\dim_{\mathbb{K}} \mathcal{L}(dP_1+P_2)=2$. From (ii)
\begin{equation} \label{in1}
 d \leq |Q|-1.
\end{equation}
Let $\alpha$ be a generator of the cyclic group $U$. Since $\alpha$ fixes both points $P_1$ and $P_2$, it acts on $\mathcal{L}(dP_1+P_2)$ as a $\mathbb{K}$-vector space automorphism $\bar{\alpha}$. If $\bar{\alpha}$ is trivial then $\alpha(u)=u$ for all $u\in\mathcal{L}(dP_1+P_2)$. Suppose that $\bar{\alpha}$ is nontrivial. Since $U$ is a prime to $p$ cyclic group, $\bar{\alpha}$ has two distinct eigenspaces, so that $\mathcal{L}(dP_1+P_2)=\mathbb{K}\oplus \mathbb{K}u$ where $u\in \mathcal{L}(dP_1+P_2)$ is an eigenvector of $\bar{\alpha}$ with eigenvalue $\xi\in \mathbb{K}^*$ so that $\bar{\alpha}(u)=\xi u$ with $\xi^{|U|}=1$. Therefore there is $u\in \mathcal{L}(dP_1+P_2)$ with $u\neq 0$ such that $\alpha(u)=\xi u$ with $\xi^{|U|}=1$.  The pole divisor
of $u$ is
\begin{equation} \label{equ1}
\Div(u)_\infty = dP_1+P_2.
\end{equation}
Since $Q=Q_{P_1}=Q_{P_2}$, the Riemann-Roch space $\mathcal{L}(dP_1+P_2)$ contains $\sigma(u)$ and hence it contains all
$$\theta_\sigma=\sigma(u)-u,\,\, \sigma\in Q.$$
By our choice of $d$ to be minimal, this yields $\theta_\sigma\in \mathbb{K}$, and then defines the map $\theta$ from $Q$ into $\mathbb{K}$ that takes $\sigma$ to $\theta_\sigma$. More precisely, $\theta$ is
a homomorphism from $Q$ into the additive group $(\mathbb{K},+)$ of $\mathbb{K}$ as the following computation shows:
$$\theta_{\sigma_1\circ\sigma_2}=(\sigma_1\circ\sigma_2)(u)-u=\sigma_1(\sigma_2(u)-u+u)-u=\sigma_1(\theta_{\sigma_2})+\sigma_1(u)-u=\theta_{\sigma_2}+\theta_{\sigma_1}=\theta_{\sigma_1}+\theta_{\sigma_2}.$$  Also, $\theta$ is injective. In fact, if $\theta_{\sigma_0}=0$ for some $\sigma_0 \in Q \setminus \{1\}$, then $u$ is in the fixed field of $\sigma_0$, which is impossible since $v_{P_2}(u)=-1$ whereas $P_2$ is totally ramified in the cover $\cX|(\cX/\langle \sigma_p \rangle)$. The image $\theta(Q)$ of $\theta$ is an additive subgroup of $\mathbb{K}$ of order $|Q|$. The smallest subfield of $\mathbb{K}$ containing $\theta(Q)$ is a finite field $\mathbb{F}_{p^m}$ and hence $\theta(Q)$ can be viewed as a linear subspace of $\mathbb{F}_{p^m}$  considered as a vector space over $\mathbb{F}_p$.  Therefore the polynomial
\begin{equation}
\label{eq210oct2016}
f(U)=\prod_{\sigma \in Q} (U-\theta_\sigma)
\end{equation}
is a linearized polynomial over $\mathbb{F}_p$; see \cite[Section 4, Theorem 3.52]{LN}. In particular, $f(U)$ is an additive polynomial of degree $|Q|$; see also \cite[Chapter V, $\mathsection \ 5$]{Serre}. Also, $f(U)$ is separable as $\theta$ is injective. From (\ref{eq210oct2016}), the pole divisor of $f(u)\in \mathbb{K}(\cX)$ is
\begin{equation}
\label{eq310oct2016}
\div(f(u))_\infty=|Q|(dP_1+P_2).
\end{equation}
For every $\sigma_0\in Q$,
$$\sigma_0(f(u))=\prod_{\sigma\in Q} (\sigma_0(u)-\theta_{\sigma})=\prod_{\sigma\in Q}(u+\theta_{\sigma_0}-\theta_{\sigma})=\prod_{\sigma\in Q}(u-\theta_{\sigma\sigma_0^{-1}})=\prod_{\sigma\in Q}(u-\theta_{\sigma})=f(u).$$
Thus $f(u) \in \mathbb{K}(\cX)^Q$. Furthermore, from $\alpha\in N_G(Q)$, for every $\sigma\in Q$ there is
$\sigma'\in Q$ such that $\alpha\sigma=\sigma'\alpha$. Therefore
$$\alpha(f(u))=\prod_{\sigma\in Q}(\alpha(\sigma(u)-u))=\prod_{\sigma\in Q}(\alpha(\sigma(u))-\xi u)=\prod_{\sigma\in Q}(\sigma'(\alpha(u))-\xi u)=\prod_{\sigma\in Q}(\sigma'(\xi u)-\xi u)=\xi f(u).$$
This shows that if $R\in\cX$ is a zero of $f(u)$ then $\supp(\div(f(u)_0))$ contains the $U$-orbit of $R$ of length $|U|$. Actually, since $\sigma(f(u))=f(u)$ for $\sigma\in Q$, $\supp(\div(f(u)_0))$ contains the $G$-orbit of $R$ of length $|G|=|Q||U|$. This together with (\ref{eq310oct2016}) give
\begin{equation}
\label{eq410oct2016}
|U|\big|(d+1).
\end{equation}
On the other hand, $\mathbb{K}(\cX)^Q$ is rational. Let $\bar{P_1}$ and $\bar{P_2}$ the points lying under $P_1$ and $P_2$, respectively, and let $\bar{R_1},\bar{R_2},\ldots, \bar{R_k}$ with $k=(d+1)/|U|$ be the points lying under the zeros of $f(u)$ in the cover $\cX|(\cX/Q)$. We may represent $\mathbb{K}(\cX)^Q$ as the projective line $\mathbb{K}\cup \{\infty\}$ over $\mathbb{K}$ so that $\bar{P_1}=\infty$, $\bar{P_1}=0$ and $\bar{R}_i=t_i$ for $1\leq i \leq k$. Let $g(t)=t^d+t^{-1}+h(t)$
where $h(t)\in \mathbb{K}[t]$ is a polynomial of degree $k=(d+1)/|U|$ whose roots are $r_1,\ldots,r_k$. It turns out that $f(u),g(t)\in \mathbb{K}(\cX)$ have the same pole and zero divisors, and hence
\begin{equation}
\label{eq510oct2016} cf(u)=t^d+t^{-1}+h(t),\,\ c\in \mathbb{K}^*.
\end{equation}
We prove that $\mathbb{K}(\cX)=\mathbb{K}(u,t)$. From \cite[Remark 12.12]{hirschfeld-korchmaros-torres2008}, the polynomial $cTf(X)-T^{d+1}-1-h(T)T$ is irreducible, and the plane curve $\mathcal{C}$ has genus $\gg(\mathcal{C})=\ha (q-1)(d+1)$. Comparison with (\ref{eq4n}) shows $\mathbb{K}(\cX)=\mathbb{\mathcal{C}}$ and $d=1$ whence $|U|\leq 2$. If equality holds then
$\deg\,h(T)=1$ and $\cX$ is a hyperelliptic curve
with equation (\ref{eq010oct2016}).
\end{proof}

The main result in this section is the following theorem.

\begin{theorem}\label{thm:sec3main}
Let $G$ be an automorphism group of an ordinary curve of even genus $g \geq 2$ defined over an algebraically closed field $K$ of characteristic $p >2$. Then $|G| < 821.37g^{7/4}$. Further, if $G\cong \rm{Alt}_7$ or $G \cong \rm{M}_{11}$, the bound can be refined to $|G| < 84(g-1)$ unless $p = 3, g = 26$, and $G \cong \rm{M}_{11}$. 
\end{theorem}

If $G$ is an odd core-free, non-solvable automorphism group of an algebraic curve of even genus $g \geq 2$ defined over a field of odd characteristic $p$, then $G$ satisfies one of the cases listed in Lemma \ref{lemA24agos2015} and Remark \ref{possibilita}. In particular, $G$ has a non-abelian minimal normal subgroup $N$ which is simple.
The proof of Theorem \ref{thm:sec3main} is divided into several steps according to the structure of $N$ and the consequent structure of $G$ in Lemma \ref{lemA24agos2015} and Remark \ref{possibilita}. 

\begin{proposition} \label{PSL2}
Let $G$ be a subgroup of $\aut(\cX)$. If $G$ admits a minimal normal subgroup $N$ which is isomorphic to $\PSL(2,q)$ for some prime power $q=d^k \geq 5$ odd, 
then $|G|< 821.37g^{7/4}$.
\end{proposition}

\begin{proof} Assume by contradiction that  
 $|G| \geq 821.37g^{7/4}$. 

According to Remark \ref{possibilita}, there are a few possibilities for the structure of $G$:
\begin{enumerate}
\item $G \cong \PSL(2,q)$,
\item $G \cong \PGL(2,q)$,
\item $G \cong \PSL(2,q) \rtimes C_r$,
\item $G \cong \PGL(2,q) \rtimes C_r$,
\end{enumerate}
where $r$ is an odd divisor of $k$ and $q \geq 5$ is odd. 

We recall that, since $q$ is odd, $|N|=q(q-1)(q+1)/2$. Moreover $N$ admits a solvable subgroup $H=S_q \rtimes C_{(q-1)/2}$, where $S_q$ is a Sylow $d$-subgroup of $N$ and $C_{(q-1)/2}$ is a cyclic group of order $(q-1)/2$, see \cite[Corollary 2.2]{king} and \cite[Theorem A.8]{hirschfeld-korchmaros-torres2008}. 

In order to apply Lemma \ref{22dic2015}, two cases are distinguished. 

\textbf{Case 1:} $|H|<30(g-1)$. In this case $(q-1)^2<q(q-1)<60(g-1)$, implying that $q<\sqrt{60(g-1)}+1$. 

Hence $|N| < 30(g-1)(\sqrt{60(g-1)}+2)=30\sqrt{60}(g-1)^{3/2}+60(g-1)<(37.75)\sqrt{60}(g-1)^{3/2}$ since $(7.75)\sqrt{60}(g-1)^{3/2}>60(g-1)$ for $g \geq 2$.

Observing that $|N|<(37.75)\sqrt{60}(g-1)^{3/2} < (292.42)g^{3/2}$ and $[\PGL(2,q):\PSL(2,q)]=2$, cases $1$ and $2$ cannot occur.

Assume that case $3$ holds. Since $r$ divides $k$ and $k=\rm{log}$$(q)/$$\rm{log}$ $(d) \leq$ $\rm{log}$ $(q)/$ $\rm{log}$$(5)\leq {\rm{log}}$$(q)/(1.6)$ , we have that 

$$(1.6)r \leq (1.6)k \leq {\rm{log}}(q) \leq {\rm{log}}(\sqrt{60(g-1)}+1) \leq \frac{(\sqrt{60(g-1)})}{(\sqrt{60(g-1)}+1)^{1/2}}< (60(g-1))^{1/4}.$$  
Thus,
\begin{equation}\label{23july171}
|G| < \bigg( \frac{(60(g-1))^{1/4}}{1.6}\bigg) \cdot (37.75)\sqrt{60}(g-1)^{3/2} < (508.64)(g-1)^{7/4},
\end{equation}
a contradiction. 

Hence we can assume that case $4$ holds. Then $q \geq 5^3$ and $|G| \geq 3|\PGL(2,5^3)|=5859000$. Since $|G| \leq 84(g^2-g)$ from Theorem \ref{thm84g2}, we have that $g \geq 266$. 

From $|N| < 30(g-1)(\sqrt{60(g-1)}+2)=30\sqrt{60}(g-1)^{3/2}+60(g-1)$ we get now that $|N|<(30.48)\sqrt{60}(g-1)^{3/2}$ since $(0.48)\sqrt{60}(g-1)^{3/2}>60(g-1)$ for $g \geq 266$. Thus,

\begin{equation}\label{23july172}
|G| <2\bigg( \frac{(60(g-1))^{1/4}}{1.6}\bigg) \cdot (30.48)\sqrt{60}(g-1)^{3/2} < (821.37)(g-1)^{7/4}< (821.37)g^{7/4}.
\end{equation}

Trivially, both Inequalities \eqref{23july171} and \eqref{23july172} give rise to a contradiction.

\textbf{Case 2:} $|H| \geq 30(g-1)$. Thus Lemma \ref{22dic2015} applies to $H$. Since $\gamma(\cX)>0$, we have that $d=p$ and $S_q$ has exactly two short orbits wich are also the unique short orbits of $H$. By Remark \ref{complemento} we have that $(q-1)/2=|C_{(q-1)/2}| < \sqrt{4g+4}=2 \sqrt{g+1}$ and so $q \leq 4 \sqrt{g+1}+1$.

As before, we get that 
$$|\PSL(2,q)| =\frac{q(q-1)(q+1)}{2}< $$
$$(4 \sqrt{g+1}+1)(2 \sqrt{g+1} )( 4 \sqrt{g+1} +2)=(2 \sqrt{g+1})(16(g+1)+12\sqrt{g+1}+2)<(47.2)(g+1)^{3/2},$$ 
as $12\sqrt{g+1}+2 < (7.6)(g+1)$ for $g \geq 2$. 

Since clearly $g+1 \leq 3g/2$, this implies that $|N|< (47.2)(g+1)^{3/2}<(86.72)g^{3/2}$ and hence both $1$ and $2$ can be excluded. 

Assume that case $3$ holds. Then, arguing as in the previous case,
$$|G| < (47.2)(g+1)^{3/2} \cdot (4 \sqrt{g+1})/(4 \sqrt{g+1}+1)^{1/2} <  (66.76)(g+1)^{7/4} <  (133)g^{7/4},$$
a contradiction.

Analogously, if case $4$ holds then
$$|G| < 2(47.2)(g+1)^{3/2} \cdot (4 \sqrt{g+1})/(4 \sqrt{g+1}+1)^{1/2} <  2(66.76)(g+1)^{7/4} \leq  (266)g^{7/4},$$
and the claim follows.
\end{proof}

\begin{proposition} \label{PSU} Let $\cX$ be an ordinary curve of even genus $g \geq 2$ defined over a field $K$ of odd characteristic $p>0$. Let $G$ be a subgroup of $\aut(\cX)$. If $G$ admits a minimal normal subgroup $N$ which is isomorphic either to $\PSU(3,q)$ or to $\PSL(3,q)$ for some prime power $q=d^k$, then $|G|< 766(g-1)^{7/4}$.
\end{proposition}

\begin{proof} Assume by contradiction that $G$ has a minimal normal subgroup $N \cong \PSU(3,q)$ (resp. $N \cong \PSL(3,q)$)  but $|G|  \geq 766(g-1)^{7/4}$. 

Let $M \cong \PGU(3,q)$ (resp. $M \cong \PGL(3,q)$). According to Remark \ref{possibilita} there are few possibilities for the structure of $G$, namely 
\begin{enumerate}
\item $G \cong N, q \equiv 1$ mod $4$ (resp. $q \equiv 3$ mod $4$) ,
\item $G \cong M,  q \equiv 1$ mod $4$ (resp. $q \equiv 3$ mod $4$),
\item $G \cong N \rtimes C_r,  q \equiv 1$ mod $4$ (resp. $q \equiv 3$ mod $4$),
\item $G \cong M \rtimes C_r,  q \equiv 1$ mod $4$ (resp. $q \equiv 3$ mod $4$),
\end{enumerate}
where $r$ is an odd divisor of $k$.  

We proceed by analyzing the cases $N \cong \PSU(3,q)$ and $N \cong \PSL(3,q)$ separately.

\textbf{Case 1:} $N \cong \PSU(3,q)$.  Let  $\delta=\gcd(3,q+1)$. Then $|N| = q^3(q^2-1)(q^3+1)/ \delta \geq |\PSU(3,5)|  = 126000$, which exceeds $84(g^2-g)$ for $g \leq 39$, whence since $g$ is even we may assume that $g \geq 40$. 

Also, $N$ admits a solvable maximal subgroup $H$ with $H=S_{q^3} \rtimes C_{(q^2-1)/ \delta}$, where $S_{q^3}$ is a Sylow $d$-subgroup of $N$ and $C_{(q^2-1)/ \delta}$ is a cyclic group of order $(q^2-1)/ \delta$, see \cite[\textsection 16]{mitchell}. 

In order to apply Lemma \ref{22dic2015}, two cases are distinguished. 
\begin{itemize}
\item Assume that $|H|<30(g-1)$, so that $(q-1)^5<q^3(q^2-1) \leq 30\delta(g-1)\leq 90(g-1)$. This in particular implies that $q <\sqrt[5]{90(g-1)}+1$. 

Thus, $$|N|=\frac{q^3(q^2-1)(q^3+1)}{\delta} <30(g-1)(q^3+1)<30(g-1)((\sqrt[5]{90(g-1)}+1)^3+1)$$
$$<30(g-1)((2\sqrt[5]{90(g-1)})^3+1)<8\cdot 30 \sqrt[5]{90}(g-1)^{8/5}+2 \sqrt[5]{90}(g-1)^{8/5}$$
$$=(242) \sqrt[5]{90}(g-1)^{8/5},$$
where the last inequality follows from $30(g-1)<2\sqrt[5]{90}(g-1)^{8/5}$ for $g \geq 40$. Since this proves that $|N|<(595.21)(g-1)^{8/5}< 345(g-1)^{7/4}$, cases $1$ and $2$ can be excluded. 

Next, assume that case $3$ or $4$ holds. Then $|N| \geq |\PSU(3,125)|$, whence $g \geq 15378928$ can be assumed from Theorem \ref{thm84g2}.  

Then, as $q$ is congruent to $1$ modulo $4$, $r \leq k \leq {\rm{log}}(q)/ {\rm{log}}(d) \leq {\rm{log}}(q)/ {\rm{log}}(5) \leq {\rm{log}}(q)/ (1.6)$. Hence,

$$(1.6)r \leq {\rm{log}}(q) \leq {\rm{log}}(\sqrt[5]{90(g-1)}+1) \leq \frac{\sqrt[5]{90(g-1)}}{(\sqrt[5]{90(g-1)}+1)^{1/2}} < \sqrt[10]{90(g-1)}.$$

Thus, 
$$|G| \leq r|\rm{PGU}(3,q)| <3 \frac{\sqrt[10]{90(g-1)}}{1.6} \cdot (595.21)(g-1)^{8/5}<(1750.24)(g-1)^{17/10}.$$

Since $(1750.24)(g-1)^{17/10}<766(g-1)^{7/4}$ for $g \geq 15378928$ we get a contradiction.

\item Hence, we can assume that $|H| \geq 30(g-1)$ and so Lemma \ref{22dic2015} applies to $H$. Since $\gamma(\cX)>0$, we have that $d=p$ and $S_{q^3}$ has exactly two short orbits wich are also the unique short orbits of $H$. 

Since $S_{q^3}$ is not abelian and a Sylow $p$-subgroup of $H$ must be elementary abelian by Theorem \ref{sylow}, we get a contradiction.

\end{itemize}

\textbf{Case 2:} $N \cong \PSL(3,q)$. As the smallest order of $\PSL(3,q)$ that can occur here is $\PSL(3,3)$ whose size is equal to $5616$, as before, $g\geq 10$ can be assumed.

Let $Q$ be a Sylow $d$-subgroup of $N$. Then $|Q| = q^3$ and by \cite[Theorem 2.4]{king} the normalizer $N_{\PSL(3,q)}(Q)$ of $Q$ in $\PSL(3,q)$ has size equal to $q^3(q-1)^2(q+1)/(3,q-1)$. In particular, for each case listed above $G$ admits a solvable subgroup of size $|N_{\PSL(3,q)}(Q)|$. 

By Lemma \ref{22dic2015}, we have that $q^3(q-1)^2(q+1)/(3,q-1) < 30(g-1)$, because otherwise $d=p$ and hence $Q$ would be elementary abelian from Theorem \ref{sylow}. Whence $q < \sqrt[6]{90(g-1)}+1$. 

From $|\PGL(3,q)|= q^3(q^3-1)(q^2-1)= (3,q-1)|\PSL(3,q)|$, we  get that if $G$ satisfies Case 1 or Case 2 then
$$
|G| \leq |\PGL(3,q)|= q^3(q^3-1)(q^2-1) =[q^3(q-1)^2(q+1)](q^2+q+1)< 90(g-1)(q^2+q+1) 
$$

$$
<180(g-1)q^2<180(g-1)(\sqrt[6]{90(g-1)}+1)^2<720(g-1)^{4/3}<290(g-1)^{7/4}, 
$$
as $g \geq 10$.

So we can assume that either Case 3 or Case 4 is satisfied. Then $|G| \leq r|\PGL(3,q)|$ where as before,
$$
(1.09)r \leq {\rm{log}}(3)r \leq {\rm{log}}(3)k \leq {\rm{log}}(q) \leq \frac{\sqrt[6]{90(g-1)}}{\sqrt{((90(g-1))^{1/6}+1)}}<\sqrt[12]{90(g-1)}.
$$
Thus, 
$$
|G| <\frac{\sqrt[12]{90(g-1)}}{1.09} \cdot 720(g-1)^{4/3}<(961.09)(g-1)^{17/12}<463(g-1)^{7/4},
$$
as $g \geq 10$, and the claim follows.
\end{proof}

For $G=N \cong \rm{Alt}_7$ and $G=N \cong \rm{M}_{11}$ we will make use of the following result, see \cite[Theorem 11.56]{hirschfeld-korchmaros-torres2008} and \cite[pagg 600-601]{nakajima 1987}.

\begin{theorem}\label{123}
Let $\cX$ be an ordinary curve, $G$ a finite subgroup of $\aut(\cX)$ and $\cY = \cX/G$. If $|G| > 84(g(\cX)-1)$, then one of the following holds.
\begin{enumerate}
\item[\rm{(i)}] $p \geq 3$ and $g(\cY) = 0$; $e_Q = 1$ if $Q \neq Q_1, Q_2, Q_3 \in \cY$, $p | e_{Q_1}$, $ e_{Q_2} = e_{Q_3} = 2$; 
\item[\rm{(ii)}] $g(\cY) = 0$; $e_Q = 1$ if $Q \neq Q_1, Q_2 \in \cY$; $p | e_{Q_1}, e_{Q_2}$;
\item[\rm{(iii)}] $g(\cY)= 0$; $e_Q = 1$ if $Q \neq Q_1, Q_2 \in \cY$; $p | e_{Q_1}$, $p \nmid e_{Q_2}$. 
\end{enumerate}
\end{theorem}

\begin{proposition}\label{alt7}
Let $G$ be a subgroup of $\aut(\cX)$. If $G = {\rm{Alt}_7}$, then $|G| <84(g-1)$, unless $p =5$ and $g(\cX) = 10$.
\end{proposition}
\begin{proof} 
Note that it is sufficient to prove that $|G| < 84(g-1)$. In fact $|G|=|{\rm{Alt}_7}|=2520=84(g-1)$ if and only if $g=31$ which is odd, and hence this case cannot occur.

If $p >7$, then ${\rm{Alt}_7}$ is a tame automorphism group of $\cX$, whence our claim follows.

 If $p = 3,5,7$, a careful case-by-case analysis is needed. Assume by contradiction that $|G|>84(g-1)$.

As $|{\rm{Alt}_7}| = 2520$, the Hurwitz genus formula applied to the covering $\cX \rightarrow \cX/G = \cY$ reads
\begin{equation}\label{eqalt7}
\frac{2g-2}{2520} = -2 + \sum_{Q \in \cY}\frac{d_Q}{e_Q},
\end{equation}
as $g(\cY) = 0$. By Theorem \ref{123}, we have to deal with cases (i)-(iii).  Let $q_1E_1$ denote the order of the stabilizer $G_{Q_1}$ of $Q_1$ in $G$, with $q_1=p^k$, $k \geq 1$ and $p \nmid E_1$. Clearly, $G_{Q_1}$ is a solvable subgroup of $G \cong {\rm Alt}_7$.

Let case (i) above hold; 
then Equation \eqref{eqalt7} yields
$$
2520 = |{\rm{Alt}_7}| = \frac{2E_1q_1}{q_1-2}(g-1),
$$
see \cite[page 601 Case (I)]{nakajima 1987}.
We claim that, regardless the characteristic, the quantity $(2E_1q_1)/(q_1-2)$ is smaller than $84$. In fact, it can be verified by MAGMA that the largest solvable subgroup of $\rm{Alt}_7$ whose Sylow $p$-subgroup (for $p \in \{3,5,7\}$) has a cyclic complement has size equal to 36, whence the assertion follows. 

Next, assume that case (ii) holds. Then \eqref{eqalt7} yields
$$
2520 = |{\rm{Alt}_7}| \leq \frac{2E_1q_1}{q_1-2}(g-1),
$$
see \cite[page 601 Case (II)]{nakajima 1987}, and the same argument as in case (i) holds. 



If (iii) holds, then  \eqref{eqalt7} reads 
$$
\frac{g-1}{1260} = -2 + \frac{d_{Q_1}}{e_{Q_1}}+ \frac{d_{Q_2}}{e_{Q_2}}.
$$
\begin{itemize}
\item Let $p = 7$. Then a $7$-Sylow subgroup of $G$ is a cyclic group $C_7$ of order $7$ whose normalizer is a semidirect product $C_7 \rtimes C_3$ with  $C_3$ a cyclic $3$-group. 
Thus, either $(e_{Q_1}, d_{Q_1}) =(7,12)$, or $(e_{Q_1}, d_{Q_1}) = (21,26)$, while $$(e_{Q_2}, d_{Q_2}) \in \{(2,1),(3,2), (4,3), (5,4), (6,5)\}.$$ Note that, in order to get $g$ even, we must have $1260\bigl( -2 + \frac{d_{Q_1}}{e_{Q_1}}+ \frac{d_{Q_2}}{e_{Q_2}}\bigr) \geq 0$ odd. By direct checking, this can happen only if $e_{Q_2} = 4$ and $e_{Q_1}=7$ with $1260\bigl( -2 + \frac{d_{Q_1}}{e_{Q_1}}+ \frac{d_{Q_2}}{e_{Q_2}}\bigr)=585$. Since this implies that $g=586$ and hence $|G|<5(g-1)$, we get a contradiction.

\item Let $p = 5$. A  $5$-Sylow subgroup of ${\rm{Alt}_7}$ is a cyclic group $C_5$ of order $5$ with normalizer $D_{10}\cdot 2$, which is isomorphic to the general affine group $\rm{GA}(1,5)$. Here, there are three possibilities for the size of a non-tame one-point stabilizer, namely $5,10,20$.

Arguing as for the previous case, we get an even value for $g$ only if $(e_{Q_1},e_{Q_2},g-1) \in \{(5,4,441), (10,4,63), (20,7,9)\}$. In the first two cases $|G|<84(g-1)$ and we get a contradiction. Thus, $g=10$, $e_{Q_1}=20$ and $e_{Q_2}=7$.

\item Let $p =3$. The $3$-Sylow of ${\rm{Alt}_7}$ is an elementary abelian group $E$ of order $9$ with normalizer $C_3\rtimes S_3\cdot 2$. 

Since we are assuming $|{\rm{Alt}_7}| \geq 84(g-1)$, we have that $g \leq 31$. If $E_1 = 1$, then by \cite[Lemma 1]{nakajima 1987} we have $|{\rm{Alt}_7}| \leq 24(g-1)$. Next, assume $E_1 \geq 2$. 

If $q_1 = 3$, then $E_1 = 2$ by \cite[Proposition 1]{nakajima 1987}. This combined with \cite[Lemma 2, Eq. 4.4]{nakajima 1987} yields 
$
|{\rm{Alt}_7}| \leq 84(g-1). 
$
If $q_1 = 9$, we are left with two possibilities: $e_{Q_1} = 18$ and $e_{Q_1} = 36$ as there is no cyclic group of order $8$ in $\rm{Alt}_7$. Then Equation \eqref{alt7} reads
$$
\frac{g-1}{1260} = -2 + \frac{25}{18}+ \frac{e_{Q_2}-1}{e_{Q_2}}, \ {\rm or} \ \ \ \frac{g-1}{1260} = -2 + \frac{43}{36}+ \frac{e_{Q_2}-1}{e_{Q_2}},
$$
with $e_{Q_2} \in \{2,4,5,7\}$ according to the possible sizes of cyclic prime-to-3 groups in $\rm{Alt}_7$. In the former case, a computation shows we can get a non-negative even genus only if $e_{Q_2} = 4$, which gives $g = 176$. In the latter case, we can have a non-negative even genus only if $e_{Q_2} = 7$, which gives $g = 66$. Both these values for $g$ are not allowed, whence our claim follows. 
\end{itemize}
\end{proof}





Next, we show that the exceptional genus $10$ case in Proposition \ref{alt7} cannot actually occur. 

\begin{proposition}\label{prop:hyp7}
Let $C_5$ be a  Sylow $5$-subgroup of $\rm{Alt}_7$. If there exists an ordinary curve $\cC$  of genus $10$ in characteristic $5$ admitting $\rm{Alt}_7$ as an automorphism group, then the quotient curve $\cC/C_5$ is isomorphic to the hyperelliptic curve 
$$
\cZ: y^4 = x(x-1)^2(x-a)^2,
$$
for $a \in K\setminus\{0,1\}$. 
\end{proposition}

\begin{proof}
Let $C_5$ be the Sylow $5$-subgroup of $\rm{Alt}_7$. Then the Riemann-Hurwitz formula applied to the covering $\cC\rightarrow \cC/C_5= \bar{\cC}$ reads
\begin{equation}\label{eq:alt7sper}
18 = 10(g(\bar{\cC})-1) + 8s,
\end{equation}

where $s$ is the number of points of $\cC$ that are fixed by $C_5$. It is easily seen that Equation \eqref{eq:alt7sper} is satisfied only for $g(\bar{\cC}) = 2$ and $s = 1$. Further, applying the Deuring-Shafarevic formula to $\cC\rightarrow \bar{\cC}$ we find
\begin{equation}\label{eq:alt7sper:ds}
9 = 5(\gamma(\bar{\cC})-1) + 4,
\end{equation}
whence $\gamma(\bar{\cC}) = 2$ easily follows. As the normalizer of $C_5$ in $\rm{Alt}_7$ is a semidirect product $C_5 \rtimes C_4$, where $C_4$ is a cyclic group of order $4$, it follows that $\bar{\cC}$ is a hyperelliptic genus $2$ curve with a cyclic automorphism group  $\bar{C}_4$ of order $4$. Let $i$ be the hyperelliptic involution of $\aut(\bar{\cC})$, and $\alpha$ a generator for $\bar{C}_4$. Then either  $\bar{C}_4 \leq  \aut(\bar{\cC})/\langle i \rangle$ or $\alpha^2 = i$. 

Assume that the former case holds. In \cite[Section 3.2]{shaska}, it is shown that over a field of characteristic $p \neq 2$, there is up to isomorphism only one curve with such an automorphism group, which is isomorphic to $\mathcal{Y}: y^2 = x^5-1$. As the $5$-rank of $\cY$ is equal to $0$, we get a contradiction.  

If the latter case holds, since we have that the genus of $\bar{\cC}/\bar{C_4}$ is equal to 0, applying the Riemann-Hurwitz formula to the covering $\bar{\cC} \rightarrow \cC/\bar{C_4}$ we get
$$
10 = 3s+2t,
$$
where $s$ is the number of points that are fixed by $\bar{C}_4$ and $t$ is the number of short orbits of size $2$. Trivially, $(s,t) = (2,2)$ is the only possible solution. By a standard argument of Kummer theory, we then have that $\bar{\cC}$ must be isomorphic to the curve 
$$
\cZ_a: y^4 = x(x-1)^2(x-a)^2,
$$
for $a \neq 0,1$, and our claim follows. It can be checked through MAGMA that, for $a \neq -1$,  $\cZ_a$ is isomorphic to a  curve with affine equation
$$
\cZ': y^2 = a_5x^5+a_3x^3+a_0,
$$
 which has $5$-rank equal to $2$. The case $a = -1$ can be discarded as $\cZ_{-1}$ has $5$-rank zero. 
\end{proof}


\begin{theorem}
There is no genus $10$ ordinary curve in characteristic $5$ with ${\rm Alt}_7$ as an automorphism group. 
\end{theorem}
\begin{proof} By contradiction, assume that such a curve $\cC$ exists. By the proof of Proposition \ref{prop:hyp7}, the quotient curve $\bar{\cC}$ is isomorphic to the hyperelliptic curve 
$$
\cZ_a: y^4 = x(x-1)^2(x-a)^2. 
$$
Note that the unique non-tame short orbit of ${\rm Alt}_7$ on $\cC$ must be partitioned into $C_4$-orbits. As  $C_4$ normalizes $C_5$ in ${\rm Alt}_7$, then $C_4$ must fix the only point that is fixed by $C_5$. Also, the two short orbits of size $3$ of the quotient group $\bar{C}$ on $\bar{\cC}$ give rise to $10$ $C_4$-short orbits on $\cC$. Applying the Riemann-Hurwitz formula to the covering $\cC\rightarrow \tilde{C} = \cC/C_4$, we get 
$$
18 = 8(\tilde{g}-1) + 20+3+\Delta,
$$
where $\tilde{g} = \g(\tilde{C})$. A computation shows $\tilde{g} = 0$ and $\Delta = 3$. In particular, $C_4$ fixes two points and has $10$ short orbits of order $2$ on $\cC$. 

 This means that the function field $K(\cC)$ is a Kummer extension of degree $4$ of a rational function field. By  a standard argument in Kummer Theory, we get, $K(\cC) = K(u,v) $ with $u,v$ satisfying 
$$
u^4 = v\prod_{i = 1}^{10}(v-a_i)^2,
$$
for pairwise distinct, non zero $a_i$s in $K$.

 Let $P_0, P_\infty$ be two places of $K(\cC)$ centered at the points $\bar{P}_0, \bar{P}_\infty$ fixed by $C_4$; also, let $$f = \frac{u^2}{v\prod_{i = 1}^{10}(v-a_i)}.$$  A computation shows  
 $$
 (f) = 2P_\infty -2P_0.
 $$
 Therefore, $\cC$ is hyperelliptic, a contradiction since ${\rm Alt}_7$ is simple and the automorphism group of a hyperelliptic curve has a central involution. 

\end{proof}

We can now state the following Corollary. 

\begin{corollary}
Let $G$ be a subgroup of $\aut(\cX)$. If $G = {\rm{Alt}_7}$, then $|G| <84(g-1)$. 
\end{corollary}


\begin{proposition}\label{m11}
Let $G$ be a subgroup of $\aut(\cX)$. If $G = {\rm{M}_{11}}$, then $|G| < 84(g-1)$ unless $p = 3$ and $g = 26$.
\end{proposition}
\begin{proof} 
We note that it is sufficient to prove that $|G| \leq 84(g-1)$ as $84$ does not divide $|G|=|\rm{M}_{11}| = 7920$.

If $p \neq 3,5,11$, then $G$ is a tame automorphism group of $\cX$, whence our claim follows. 

 If $p = 3,5,11$, a careful case-by-case analysis is needed.  Assume by contradiction that $|G|>84(g-1)$.

As $|\rm{M}_{11}| = 7920$, the Hurwitz genus formula applied to the covering $\cX \rightarrow \cX/G = \cY$ reads  
\begin{equation} \label{eqm11}
\frac{g-1}{3960} = -2 + \sum_{Q \in \cY}\frac{d_Q}{e_Q}.
\end{equation}

If case (i) or (ii) holds, arguing as in Proposition \ref{alt7} (and keeping the same notation) we have 
$$
7920 = |{\rm{M}_{11}}| \leq \frac{2E_1q_1}{q_1-2}(g-1). 
$$
Again, the quantity $(2E_1q_1)/(q_1-2)$ is smaller than $84$ as $(E_1q_1,q_1) \in \{(72,9),(55,11)\}$ or $E_1q_1 \leq 36$. 

Next, assume that case (iii) holds; then \eqref{eqm11} reads
$$
\frac{g-1}{3960} = -2 + \frac{d_{Q_1}}{e_{Q_1}}+ \frac{d_{Q_2}}{e_{Q_2}}.
$$
\begin{itemize}
\item Let $p = 11$. A Sylow $11$-subgroup of $\rm{M}_{11}$ is a cyclic group $C_{11}$ of order $11$ whose normalizer is the semidirect product $C_{11} \rtimes C_5$ where  $C_5$ is a cyclic $5$-group, 
with $(e_{Q_1}, d_{Q_1}) \in \{(11,20), (55,64)\}$, while $(e_{Q_2}, d_{Q_2}) \in \{(2,1),(3,2), (4,3), (5,4),(6,5), (8,7)\}.$ 

Then Equation \eqref{m11} reads
$$
\frac{g-1}{3960} = -2 + \frac{20}{11}+ \frac{e_{Q_2}-1}{e_{Q_2}}, \ \ {\rm{or}} \ \  \frac{g-1}{3960} = -2 + \frac{64}{55}+ \frac{e_{Q_2}-1}{e_{Q_2}}.
$$
The only values for $e_{Q_2}$ and $e_{Q_1}$ yielding a non-negative even genus are $e_{Q_2} = 8$ and $e_{Q_1}=11$, which give $g = 2746$ and $|G|<84(g-1)$, a contradiction.

\item Let $p = 5$. A Sylow $5$-subgroup of $\rm{M}_{11}$ is a cyclic group $C_5$ of order $5$ with normalizer $D_{10}\cdot 2$, which is isomorphic to the general affine group $\rm{GA}(1,5)$. 

Here, there are three possibilities for the size of a non-tame one-point stabilizer, namely $5,10,20$. Arguing as before, to get a non-negative even value for $g$, the only possibilities are $(e_{Q_1},e_{Q_2},g-1) \in \{(5,8,1881), (10,8,693), (20,8,99) \}$. Since $|G|=7920>84(g-1)$ implies $g \leq 94$ we get a contradiction.

\item Let $p = 3$. Then the Sylow $3$-subgroup of $\rm{M}_{11}$ is an elementary abelian group $E$ of order $9$ whose normalizer is the semidirect product $E \rtimes SD_{16}$ where $SD_{16}$ is a semidihedral group of order 16. 


 In this case, $(e_{Q_2}, d_{Q_2}) \in \{(2,1), (4,3), (5,4), (8,7), (11,10)\}$ 
while $(e_{Q_1},d_{Q_1}) \in \{(18,25),$ $ (9,16), (6,7), (72,79), (36,43), (3,4)\}$.
Note that, apart from the case $(e_{Q_2}, d_{Q_2})=(8,7)$, in all the other cases we get an even genus if and only if $(e_{Q_1}, d_{Q_1}) = (72,79)$. If this is the case, the above formula yields a nonnegative value if and only if $(e_{Q_2}, d_{Q_2}) = (11,10)$, which yields $g = 26$. We now prove that the case $(e_{Q_2}, d_{Q_2}) = (8,7)$ is impossible. In fact if $(e_{Q_2}, d_{Q_2}) = (8,7)$ then $g \geq 2$ is even if and only if $(e_{Q_1},g-1) \in \{(18,1045), (9,2585), (6,165), (36,275), (3,825)\}$. Since $g >94$, we have that $7920=|G|<84(g-1)$, and hence in all these cases we get a contradiction.

\end{itemize}
\end{proof}

\section{On ordinary curves of genus 26 admitting $\rm{M}_{11}$ when p =3}\label{sec:modular}
In this Section, the characteristic of the ground field $K$ is assumed to be $p =3$. We deal with the existence of ordinary genus $26$ curves with $\rm{M}_{11}$ as an automorphism group. 
For instance,  the modular curve $X(11)$ has genus $26$ and in characteristic $3$, one has $\aut(X(11)) \cong \rm{M}_{11}$; see  \cite{adler 1997, rajan}. A plane model for $X(11)$ is given by 
$$
X(11): y^{10}(y+1)^9= x^{22}-y(y+1)^4x^{11}(y^3+2y+1),
$$
see \cite[Section 4.3]{yy}. 
By direct MAGMA computation it can be checked that  $\gamma(X(11)) = g(X(11)) = 26$.
Further, the following holds. 

\begin{proposition}
Let $\mathcal{X}$ be an ordinary curve of genus $26$ such that $\rm{M}_{11} \subseteq \aut(\mathcal{X})$. Denote by $E$ the Sylow $3$-subgroup of $\rm{M}_{11}$. Then the quotient curve $\cX/E$ is isomorphic to the hyperelliptic genus $2$ ordinary curve 
$$
\cY: y^2 = x^5-x. 
$$
\end{proposition}

\begin{proof}
By the proof of Proposition \ref{m11}, $\rm{M}_{11}$ can have only two short orbits on $\cX$: a non-tame orbit $O_1$, formed by $110$ points, each with ramification index $72$ and a tame orbit $O_2$ consisting of 720 points each  with ramification index $11$.
Recall that $E$ is an elementary abelian $3$-group of order $9$. 
There are exactly $55$ Sylow $3$-subgroups in $\rm{M}_{11}$. 

The ramified points  of $\cX \rightarrow \cX/E$ are contained in $O_1$. Moreover, $O_1$ can be partitioned into short orbits of the Sylow $3$-subgroups. A counting argument shows that each Sylow $3$-subgroup fixes two points,   $P_1, P_2$.

 For $j = 1,2$,  we have $|E^{(0)}_{P_j}| = |E^{(1)}_{P_j}| = 9$. 
Hence, the Hurwitz genus formula yields $g(\cX/E )= 2$.  
As the normalizer of $E$ in $\rm{M}_{11}$ is the semidirect product $E \rtimes SD_{16}$, then $\cX/E$ has an automorphism group isomorphic to $SD_{16}$.

In \cite[Section 3.2]{shaska}, it is shown that over a field of characteristic $p \neq 2$, there is up to isomorphism only one curve with such an automorphism group, which is isomorphic to $\mathcal{Y}$. The full automorphism group of $\mathcal{Y}$ is $\tilde{S}_4$, a double cover of $S_4$.  
\end{proof}


We end this section with an open question. 

\begin{question}
Is $X(11)$ the only ordinary genus $26$ curve in characteristic $3$ admitting $\rm{M}_{11}$ as an automorphism group up to isomorphism?  
\end{question}

\subsection*{Acknowledgments}
 This research was performed within the activities of  GNSAGA - Gruppo Nazionale per le Strutture Algebriche, Geometriche e le loro Applicazioni of Italian INdAM. The second author  was supported by FAPESP-Brazil, grant 2017/18776-6.

\end{document}